\documentclass[oneside,english]{amsart}
\usepackage[T1]{fontenc}
\usepackage[latin9]{inputenc}
\usepackage{geometry}
\usepackage{textcomp}
\usepackage{amstext}
\usepackage{amsthm}

\makeatletter
\numberwithin{equation}{section}
\numberwithin{figure}{section}
\theoremstyle{plain}
\newtheorem{thm}{\protect\theoremname}[section]
\theoremstyle{plain}
\newtheorem{conjecture}[thm]{\protect\conjecturename}
\theoremstyle{plain}
\newtheorem{cor}[thm]{\protect\corollaryname}
\theoremstyle{plain}
\newtheorem{lem}[thm]{\protect\lemmaname}
\theoremstyle{remark}
\newtheorem{rem}[thm]{\protect\remarkname}
\theoremstyle{plain}
\newtheorem{prop}[thm]{\protect\propositionname}

\def\makebbb#1{
    \expandafter\gdef\csname#1\endcsname{
        \ensuremath{\Bbb{#1}}}
}\makebbb{R}\makebbb{N}\makebbb{Z}\makebbb{C}\makebbb{H}\makebbb{E}\makebbb{H}\makebbb{P}\makebbb{B}\makebbb{Q}\makebbb{E}\makebbb{E}
\makebbb{H}
\usepackage{babel}

\usepackage{babel}

\makeatother

\usepackage{babel}

\makeatother

\usepackage{babel}
\providecommand{\conjecturename}{Conjecture}
\providecommand{\corollaryname}{Corollary}
\providecommand{\lemmaname}{Lemma}
\providecommand{\propositionname}{Proposition}
\providecommand{\remarkname}{Remark}
\providecommand{\theoremname}{Theorem}

\begin{document}
\title{The probabilistic vs the quantization approach to Kähler-Einstein
geometry }
\author{Robert J. Berman}
\begin{abstract}
In the probabilistic construction of Kähler-Einstein metrics on a
complex projective algebraic manifold $X$ - involving random point
processes on $X$ - a key role is played by the partition function.
In this work a new quantitative bound on the partition function is
obtained. It yields, in particular, a new direct analytic proof that
$X$ admits a Kähler-Einstein metrics if it is uniformly Gibbs stable.
The proof makes contact with the quantization approach to Kähler-Einstein
geometry.
\end{abstract}

\maketitle

\section{\label{sec:Introduction}Introduction}

A complex projective algebraic manifold $X$ admits a Kähler-Einstein
metric with positive Ricci curvature if and only if $X$ is a Fano
manifold satisfying an algebro-geometric condition called K-stability;
this is the content of the solution of the Yau-Tian-Donaldson (YTD)
conjecture for Fano manifolds \cite{c-d-s}. The proof in \cite{c-d-s}
is based on a variant of Aubin's method of continuity \cite{au},
extended to Aubin's original method in \cite{d-s}. It involves the
following equations for a Kähler metric $\omega_{t},$ parameterized
by ``time'' $t:$
\begin{equation}
\mbox{\ensuremath{\mbox{Ric}}\ensuremath{\,\omega_{t}}}=t\omega_{t}+(1-t)\ensuremath{\mbox{Ric}\,}dV,\label{eq:Aubin eq}
\end{equation}
where $dV$ is a fixed a volume form $dV$ on $X,$ which may be taken
to have positive Ricci curvature $\ensuremath{\mbox{Ric}\,}dV$ (since
$X$ is Fano). The supremum over all $t\in[0,1]$ for which a solution
$\omega_{t}$ exists defines an invariant of $X,$ denoted by $R(X),$
which is strictly positive \cite{sz}. As $t$ is increased towards
$R(X)$ either $\omega_{t}$ blows-up or it converges towards a Kähler-Einstein
metric (in which case $R(X)=1).$ The first alternative is precisely
what it is shown to be excluded by the condition of K-stability \cite{d-s}.
While it is usually assumed that $t\in[0,1]$ it will in the present
work be important to allow $t$ to be any real number.

A probabilistic construction of Kähler-Einstein metrics with negative
Ricci curvature was introduced in \cite{berm8}, where the Kähler-Einstein
metric emerges from a random point process on $X$ with $N$ points
as $N$ tends to infinity (see also \cite{f-k-z} for a different
probabilistic framework involving random $N\times N$ Hermitian matrices,
also inspired by the YTD conjecture). A conjectural extension to Kähler-Einstein
metrics with positive Ricci curvature was proposed in \cite{berm8 comma 5}
and conditional convergence results were given in \cite{berm11,berm12}.
In this probabilistic approach the role of K-stability is played by
a new type of stability, dubbed Gibbs stability, which amounts to
the finiteness of the corresponding partition functions. In the survey
\cite{berm11b} connections to the variational proof of the uniform
YTD conjecture \cite{bbj} (involving uniform K-stability) are explained,
including non-Archimedean aspects. In the present paper a new quantitative
lower bound on the partition functions is obtained, which  yields
a new direct analytic proof that uniform Gibbs stability implies the
existence of a unique Kähler-Einstein metric on $X.$ The proof makes
contact with the quantization approach to Kähler geometry and, in
particular, with K.Zhang's new remarkably direct proof of the (uniform)
YTD conjecture \cite{zh}. 

\subsection{Background on the probabilistic approach}

Let $X$ be a Fano manifold. Given a positive integer $k$ we denote
by $N$ the dimension of the space of all holomorphic sections of
the $k$th tensor power of the anti-canonical line bundle $-K_{X}$
(i.e. the top exterior power of the tangent bundle of $X)$: 
\[
N:=\dim H^{0}(X,-kK_{X})
\]
(using additive notation for tensor products of line bundles). The
Fano assumption on $X$ ensures, in particular, that $N\rightarrow\infty,$
as $k\rightarrow\infty$ (more precisely, $N\sim k^{\dim X}$). Given
a basis $s_{1}^{(k)},...,s_{N}^{(k)}$ in $H^{0}(X,-kK_{X})$ denote
by $\det S^{(k)}$ the corresponding holomorphic section of the line
bundle $-(kK_{X^{N}})\rightarrow X^{N}$ defined as the Slater determinant
\begin{equation}
(\det S^{(k)})(x_{1},x_{2},...,x_{N}):=\det\left(s_{i}^{(k)}(x_{j})\right).\label{eq:slater determinant}
\end{equation}
Given a volume form $dV$ on $X$ and a parameter $\beta>0,$ the
$N-$fold product $X^{N}$ is endowed with the following probability
measure:

\[
\mu_{\beta}^{(N)}:=\frac{\left\Vert \det S^{(k)}\right\Vert ^{2\beta/k}dV^{\otimes N}}{\mathcal{Z}_{N}(\beta)},\,\,\,\mathcal{Z}_{N}(\beta):=\int_{X^{N_{k}}}\left\Vert \det S^{(k)}\right\Vert ^{2\beta/k}dV^{\otimes N}
\]
where $\left\Vert \cdot\right\Vert $ denotes the metric on $-K_{X}$
(and its tensor powers) induced by $dV.$ In statistical mechanical
terms this probability measure represents the equilibrium distribution
of $N$ interacting particles on $X$ at inverse temperature $\beta$
and $\mathcal{Z}_{N_{k}}(\beta)$ is the corresponding partition function.
The probability measure $\mu_{\beta}^{(N)}$ is, in fact, independent
of the choice of bases. It will be convenient to fix a reference volume
form $dV_{X}$ on $X$ with positive Ricci curvature and a basis $(s_{i}^{(k)})$
in $H^{0}(X,-kK_{X})$ which is orthonormal with respect to Hermitian
product on $H^{0}(X,-kK_{X})$ induced by $dV_{X}.$ 

The probability measure $\mu_{\beta}^{(N)}$ is symmetric (since the
determinant is anti-symmetric) and thus defines a random point process
on $X$ with $N$ points $x_{1},...,x_{N}.$ By \cite[Thm 5.7]{berm8}
the corresponding \emph{empirical measure} $\delta_{N},$ i.e. the
discrete measure on $X$ defined by 
\begin{equation}
\delta_{N}:=\frac{1}{N}\sum_{i=1}^{N}\delta_{x_{i}},\label{eq:def of empirical measure}
\end{equation}
converges in probability, as $N\rightarrow\infty,$ towards a normalized
volume form $dV_{\beta}$ on $X$ with the property that the Kähler
form 
\begin{equation}
\omega_{\beta}:=\frac{1}{\beta}\frac{i}{2\pi}\partial\bar{\partial}\log dV_{\beta}\label{eq:def of omega beta}
\end{equation}
is the unique solution to Aubin's continuity equation \ref{eq:Aubin eq}
with $t:=-\beta.$ The convergence of $\delta_{N}$ towards $dV_{\beta}$
also implies that the following convergence holds in the weak topology
of currents on $X:$ 
\[
\omega_{k,\beta}:=\frac{1}{\beta}\frac{i}{2\pi}\partial\bar{\partial}\left(\log\int_{X^{N-1}}\left\Vert \det S^{(k)}\right\Vert ^{2\beta/k}dV^{\otimes N-1}\right)\rightarrow\omega_{\beta},\,\,\,k\rightarrow\infty,
\]
 where $\omega_{k,\beta}$ is a Kähler form, for $k$ sufficiently
large (to ensure that $-kK_{X}$ is very ample). 

In fact, the convergence of $\delta_{N}$ towards $dV_{\beta}$ was
shown to hold at an exponential speed in the sense of Large deviation
theory \cite{d-z}. More precisely, a Large Deviation Principle (LDP)
was established, which may be symbolically expressed as 
\begin{equation}
(\delta_{N})_{*}\left(\left\Vert \det S^{(k)}\right\Vert ^{2\beta/k}dV^{\otimes N}\right)\sim e^{-NF_{\beta}(\mu)},\,\,\,N\rightarrow\infty,\label{eq:LDP}
\end{equation}
 where the left hand side defines a measure on the space of all probability
measures $\mathcal{P}(X)$ on $X$ and $F_{\beta}(\mu)$ is a free
energy type functional on $\mathcal{P}(X)$ (see formula \ref{eq:def of free en}).
Expressing $\mu$ as the normalized volume form of a Kähler metric
$\omega$ in the space $\mathcal{H}$ of all Kähler metrics representing
the first Chern class of $X,$ the free energy functional $F_{\beta}(\mu)$
gets identified with the twisted Mabuchi functional on $\mathcal{H}$
(which is minimized precisely by the unique Kähler metrics $\omega_{\beta}$
solving Aubin's equation \ref{eq:Aubin eq} with $t=-\beta):$ 
\[
F_{\beta}(\frac{\omega^{n}}{V})=\mathcal{M}_{\beta}(\omega)
\]
 (cf. formula \ref{eq:M as free energy}). The LDP \ref{eq:LDP} thus
implies that 
\begin{equation}
\lim_{N\rightarrow\infty}-\frac{1}{N}\log\mathcal{Z}_{N}(\beta)=\inf_{\mathcal{H}}\mathcal{M}_{\beta}\label{eq:lim Z N is inf Mab intro}
\end{equation}

\subsubsection{The case $\beta<0$ }

In the case when $\beta$ is negative the probability measure $\mu_{\beta}^{(N)}$
is well-defined for $\beta$ sufficiently close to $0.$ The main
case of interest is when $\beta=-1.$ In this case the measure $\mu_{\beta}^{(N)}$
is canonically attached to $X$, i.e. it is independent of the choice
of volume form $dV$ (since the contributions from the metric $\left\Vert \cdot\right\Vert $
on $-K_{X}$ and the volume form $dV$ on $X$ cancel). Hence, if
$\mu_{-1}^{(N)}$ is well-defined when $N$ is sufficiently large,
i.e. if $\mathcal{Z}_{N}(-1)<\infty$ - in which case $X$ is called
\emph{Gibbs stable} - one obtains canonical random point processes
on $X$ with $N$ points. It was conjectured in \cite{berm8 comma 5}
that the corresponding empirical measures $\delta_{N}$ converge towards
a unique Kähler-Einstein metric on $X$ \cite{berm8 comma 5}, as
$N\rightarrow\infty.$ A conjectural extension of the LDP for positive
$\beta$ in formula \ref{eq:LDP} to any negative $\beta$ was also
put forth in \cite{berm8 comma 5}. In a weaker form this conjecture
may be formulated as follows: 
\begin{conjecture}
\label{conj:Let--be}Let $X$ be a Fano manifold endowed with a volume
form $dV.$ Given a negative number $\beta_{0}$ the following is
equivalent:
\begin{enumerate}
\item For any given $\beta>\beta_{0}$ the partition function $\mathcal{Z}_{N}(\beta)$
is finite, when $N$ is sufficiently large.
\item For any given $\beta>\beta_{0}$ the twisted Mabuchi functional $\mathcal{M}_{\beta}$
admits a minimizer in $\mathcal{H}.$ 
\end{enumerate}
Moreover, if $\beta_{0}$ satisfies the first condition, then for
any given $\beta>\beta_{0}$ the empirical measure $\delta_{N}$ of
the ensemble $\left(X^{N},\mu_{\beta}^{(N)}\right)$ converges in
probability as $N\rightarrow\infty$ - after perhaps passing to a
subsequence - towards a volume form $dV_{\beta}$ such that the corresponding
Kähler metric $\omega_{\beta}$ (formula \ref{eq:def of omega beta})
minimizes $\mathcal{M}_{\beta}$ on $\mathcal{H}.$ 
\end{conjecture}

The reason that one has to pass to a subsequence in the conjectural
convergence above is that a minimizer of $\mathcal{M}_{\beta}$ need
not be uniquely determined, unless $dV$ is assumed to have positive
Ricci curvature and $\beta>-1.$ The integrability condition $\mathcal{Z}_{N_{k}}(\beta)<\infty$
is, however, independent of choice of volume form $dV.$ Accordingly,
one obtains invariants of the Fano manifold $X$ by setting
\begin{equation}
\gamma_{k}(X)=\sup_{\gamma>0}\left\{ \gamma:\,\,\,\mathcal{Z}_{N_{k}}(-\gamma)<\infty\right\} \,\,\,\gamma(X):=\liminf_{k\rightarrow\infty}\gamma_{k}(X).\label{eq:def of gamma k and gamma}
\end{equation}
and $X$ is called \emph{uniformly Gibbs stable if $\gamma>1.$ }This
is, a priori, a stronger condition than Gibbs stability (which amounts
to the condition that $\gamma_{k}(X)>1$ for any sufficiently large
$k).$\emph{ }

The validity of the equivalence $"1\iff2"$ in the previous conjecture
would, in particular, imply that a Fano manifold $X$ is uniformly
Gibbs stable iff $X$ admits a unique Kähler-Einstein metric (in analogy
to the uniform version of YTD \cite{bbj}, which, in fact, is equivalent
to the ordinary formulation of the YTD \cite{l-x-z}). The general
equivalence $"1\iff2"$ may be reformulated as the following identity:
\begin{equation}
\gamma(X)=\sup_{\gamma>0}\left\{ \gamma:\,\,\,\inf_{\mathcal{H}}\mathcal{M}_{-\gamma}>-\infty\right\} ,\label{eq:gamma equals Gamma}
\end{equation}
as follows from \cite[Thm 1.2]{berm6} (when restricted to $\beta\geq-1$
the supremum in the right hand side above coincides with the maximal
existence time $R(X)$ for Aubin's equations \ref{eq:Aubin eq}).
Moreover, as shown in \cite[Section 7]{berm11} and \cite[Thm 2.3]{berm12},
in order to prove the conjectured convergence towards a minimizer
of $\mathcal{M}_{\beta}$ it is enough to extend the asymptotics \ref{eq:lim Z N is inf Mab intro}
to $\beta<0.$

\subsection{Main results}

For $\beta<0$ the limsup upper bound in formula \ref{eq:lim Z N is inf Mab intro}
was established in \cite[Thm 6.7]{berm8 comma 5} (by combining Gibbs
variational principle in statistical mechanics with the asymptotics
for transfinite diameters in \cite[Thm 6.7]{b-b}). The main new result
in the present work is the following quantitative upper bound that
holds for any fixed $k$, shown using a completely different argument.
Henceforth, we set $\gamma:=-\beta$.
\begin{thm}
\label{thm:Main intro}There exists a constant $C>0$ (depending only
on the reference volume for $dV_{X})$ such that for any $\gamma>0$
and positive integer $k$
\[
-\frac{1}{N}\log\mathcal{Z}_{N}(-\gamma)\leq\frac{k+\gamma}{k+1}\inf_{\mathcal{H}}\mathcal{M}_{-\gamma c_{k}}+k^{-1}\gamma\left(C+\left(|1-\gamma|+C\right)\log\left\Vert \frac{dV}{dV_{X}}\right\Vert _{L^{\infty}(X)}\right),
\]
 where $c_{k}:=(1-Ck^{-1})(k+1)/(k+\gamma).$
\end{thm}

For $\gamma\leq1$ the first term in the right hand side of the previous
inequality may be replaced by the infimum of $\mathcal{M}_{-\gamma(1-Ck^{-1})}$
(see Section \ref{subsec:The-case gamma smaller}).

The previous theorem immediately implies one direction of the conjectured
equality \ref{eq:gamma equals Gamma}:
\begin{cor}
\label{cor:intro}The following inequality holds
\[
\gamma(X)\leq\sup_{\gamma>0}\left\{ \gamma:\,\,\,\inf_{\mathcal{H}}\mathcal{M}_{-\gamma}>-\infty\right\} 
\]
In other words $"1\implies2"$ in Conjecture \ref{conj:Let--be}.
In particular, if $X$ is uniformly Gibbs stable, then $X$ admits
a unique Kähler-Einstein metric.
\end{cor}

As next explained this corollary also follows from combining the algebro-geometric
results in \cite[Thm 6.7]{f-o} with the solution of the (uniform)
YTD-conjecture in \cite[Thm 6.7]{c-d-s} (or \cite[Thm 6.7]{bbj})
and its very recent generalization in \cite[Thm 6.7]{zh} (which applies
to general $\beta).$ More precisely, exploiting that $\gamma_{k}(X)$
may be realized as the log canonical threshold (lct) of an anti-canonical
divisor on $X^{N_{k}}$ it is shown in \cite[Thm 2.5]{f-o} that $\gamma_{k}(X)$
is bounded from above by the invariant $\delta_{k}(X)$ introduced
in \cite{f-o}: 
\begin{equation}
\gamma_{k}(X)\leq\delta_{k}(X):=\inf_{\Delta_{k}}\text{lct }(\Delta_{k}),\label{eq:gamma k smaler then delta k}
\end{equation}
 where the infimum is taken over all anti-canonical $\Q-$divisors
$\Delta_{k}$ on $X$ of $k-$basis type, i.e. $\Delta_{k}$ is the
normalized sum of the $N$ zero-divisors on $X$ defined by the members
of a given basis in $H^{0}(X,-kK_{X})$. In particular, 
\[
\gamma(X)\leq\delta(X):=\limsup_{k\rightarrow\infty}\delta_{k}(X),
\]
 where the invariant $\delta(X)$ characterizes uniform K-stability;
$\delta(X)>1$ iff $X$ is uniformly K-stable \cite{fu2}. Recently,
it was shown in \cite{rtz} that $\delta_{k}(X)$ coincides with the
coercivity threshold of the quantized Ding functional on the symmetric
space $GL(N,\C)/U(N).$ Combining this result with the quantized maximum
principle in \cite{bern}, it was then shown in \cite{zh} that $\delta(X)$
coincides with the coercivity threshold of the Ding functional (as
further discussed in Section \ref{subsec:Comparison-with-the qu}).
Finally, Cor \ref{cor:intro} follows from \cite[Thm 3.4]{berm6},
which implies that the coercivity thresholds of the Ding and the Mabuchi
functionals coincide.

\subsubsection{Outline of the proof of Theorem \ref{thm:Main intro}}

The proof of Theorem \ref{thm:Main intro} is surprisingly simple.
The key new observation is an inequality which, in its simplest form,
$\beta=-1$ (i.e. $\gamma=1),$ may be formulated as follows: 
\begin{equation}
-\log\mathcal{Z}_{N}\leq\left(1+k^{-1}\right)\inf_{\mathcal{H}}\mathcal{D}_{k}+\frac{1}{kN}\log N\label{eq:ineq between log Z N and first Ding intro}
\end{equation}
 where the infimum runs over the space $\mathcal{H}$ of all metrics
on $-K_{X}$ with positive curvature and $\mathcal{D}_{k}$ is a certain
(scale invariant) functional on $\mathcal{H},$ approximating the
twisted Ding functional $\mathcal{D}$ (in the sense that $\mathcal{D}_{k}$
converges towards $\mathcal{D}$ as $k\rightarrow\infty);$ see formula
\ref{eq:def of D k gamma on H}. Next, by an inequality established
in \cite{b-f} (leveraging the positivity of direct image bundles
in \cite{bern0}) there exists a constant $C$ such that
\[
\mathcal{D}_{k}\leq\mathcal{D}-Ck^{-1}\mathcal{E}\,\,\,\text{on \ensuremath{\mathcal{H}_{0}}},
\]
 where $\mathcal{H}_{0}$ denotes the subspace of all sup-normalized
metrics on $\mathcal{H}$ and $\mathcal{E}$ denotes the standard
functional on $\mathcal{H}$ defined as the primitive of the Monge-Ampère
operator (which is non-positive on $\mathcal{H}_{0}).$ Finally, using
the well-known fact that $\mathcal{D}$ is bounded from above by the
Mabuchi functional $\mathcal{M}$ this proves Theorem \ref{thm:Main intro}
when $\gamma=1$ (by absorbing the error term $-Ck^{-1}\mathcal{E}$
in the subscript $\gamma$ of the twisted Mabuchi functional $\mathcal{M}_{\gamma}).$
A slight twist of this argument yields the inequality in Theorem \ref{thm:Main intro}
for a general $\gamma,$ using the thermodynamical formalism in \cite{berm6}. 

\subsection{\label{subsec:Comparison-with-the qu}Comparison with the quantization
approach}

In the quantization approach to Kähler geometry, which goes back to
\cite{y2,ti00,do0,do1}, the space $\mathcal{H}(L)$ of all Hermitian
metrics on a holomorphic line bundle $L$ over a complex manifold
$X$ is approximated by the finite dimensional space $\mathcal{H}_{k}(L)$
of all Hermitian metrics on the $N-$dimensional complex vector space
$H^{0}(X,kL)$ \footnote{In physical terms $\mathcal{H}_{k}(L)$ can be viewed as the quantization
of $\mathcal{H}$ with $k^{-1}$ playing the role of Planck's constant
in quantum mechanics \cite{do0}.} The space $\mathcal{H}_{k}(L)$ may be identified with the symmetric
space $GL(N,\C)/U(N).$ When $X$ is Fano and $L=-K_{X}$ a quantization
of the Ding functional $\mathcal{D}$ on $\mathcal{H}$ was introduced
in \cite{bbgz}, building on \cite{do1c}, which defines a functional
on $\mathcal{H}_{k}$ that we shall denote by $D_{k}$ (formula \ref{eq:def of D k gamma on H k}).
Here it is observed (Prop \ref{prop:inequal for quantized Ding})
that

\begin{equation}
\inf_{\mathcal{H}_{k}}D_{k}=\left(1+k^{-1}\right)\inf_{\mathcal{H}}\mathcal{D}_{k},\label{eq:inf D k equals inf D k intro}
\end{equation}
where $\mathcal{D}_{k}$ is the approximation on $\mathcal{H}$ of
the Ding functional $\mathcal{D}$ which appeared in the inequality
\ref{eq:ineq between log Z N and first Ding intro}. As a consequence,
\begin{equation}
-\log\mathcal{Z}_{N}\leq\inf_{\mathcal{H}_{k}}D_{k}+\frac{1}{kN}\log N.\label{eq:ZN in terms of inf intro}
\end{equation}
A similar inequality holds for a general $\gamma$ (see Theorem \ref{thm:bound on Z in terms of D on H k})
which yields a new proof of the inequality \ref{eq:gamma k smaler then delta k}. 

This line of reasoning is inspired by K.Zhang's very recent new proof
of the uniform YTD conjecture for Fano manifolds \cite{zh}. In fact,
the author discovered the equality \ref{eq:inf D k equals inf D k intro}
while trying to find a conceptual replacement for an inequality used
in the proof of \cite[Thm 5.1]{zh} (involving Tian's $\alpha-$invariant
\cite{ti000}). One virtue of the present approach is that it directly
yields a quantitative estimate on the infimum of $D_{k,\beta}$ over
$\mathcal{H}_{k}:$
\begin{equation}
\inf_{\mathcal{H}_{k}}D_{k,-\gamma}\leq\frac{k+\gamma}{k+1}\inf_{\mathcal{H}}\mathcal{M}_{-\gamma c_{k}}+k^{-1}\gamma\left(C+\left(|1-\gamma|+C\right)\log\left\Vert \frac{dV}{dV_{X}}\right\Vert _{L^{\infty}(X)}\right)\label{eq:inf D k over H k smaller than Mab intro}
\end{equation}
 by combining formula \ref{eq:inf D k equals inf D k intro} (extended
to general $\gamma)$ with the inequality \ref{eq:ineq between log Z N and first Ding intro}
(extended to general $\gamma)$. As in \cite[Thm 5.1]{zh} this shows
that uniform K-stability of $X$ implies that $X$ admits a unique
Kähler-Einstein metric. Indeed, as shown in \cite{rtz}, building
on \cite{f-o,fu2}, uniform K-stability is equivalent to the existence
of some $\epsilon>0$ such that the infimum of $D_{k,-1-\epsilon}$
on $\mathcal{H}_{k}$ is finite for $k$ sufficiently large. By the
inequality \ref{eq:inf D k over H k smaller than Mab intro} this
implies that $\mathcal{M}_{-1-\epsilon}$ is bounded from below (or
equivalently, that $\mathcal{M}_{-1}$ is coercive) which, in turn,
implies that $X$ admits a unique Kähler-Einstein metric (as first
shown in \cite{ti} using Aubin's method of continuity and then using
a direct variational approach in \cite{berm6,bbegz}, which applies
to any $\gamma$).

\subsection{Outlook on converse bounds and exceptional Fano orbifolds}

The converse of the inequality \ref{eq:inf D k over H k smaller than Mab intro}
also holds, as follows from \cite[Lemma 7.7]{bbgz}. As a consequence,
\begin{equation}
\delta(X)=\sup_{\gamma>0}\left\{ \gamma:\,\,\,\inf_{\mathcal{H}}\mathcal{M}_{-\gamma}>-\infty\right\} .\label{eq:delta is inf M}
\end{equation}
This identity is equivalent to the result \cite[Thm 5.1]{zh} (which
is formulated in terms of the infimum of $\mathcal{D}_{\beta}$, but,
by \cite[Thm 1.1]{berm6}, this infimum coincides with the infimum
of $\mathcal{M}_{\beta}).$ It remains, however, to establish a similar
lower bound on $-\log\mathcal{Z}_{N,-\gamma}$ or, at least, the missing
lower bound on $\gamma(X)$ in the conjectured formula \ref{eq:gamma equals Gamma}.
By formula \ref{eq:delta is inf M}, this amounts to upgrading the
inequality between $\gamma(X)$ and $\delta(X)$ in Cor \ref{cor:intro}
to an equality. In contrast, it should be stressed that the inequality
\ref{eq:gamma k smaler then delta k} between $\gamma_{k}(X)$ and
$\delta_{k}(X)$ is \emph{not }an equality, in general. For example,
when $X$ is the Riemann sphere, i.e. the complex projective line
$\P^{1},$ 
\[
\gamma_{k}(X)=1-\frac{1}{2k+1},\,\,\,\delta_{k}(X)=1.
\]
\cite{fu,rtz}. This discrepancy becomes even more pronounced in the
more general of setting of Fano orbifolds $X,$ where the role of
$K_{X}$ is played by the orbifold canonical line bundle $K_{X_{orb}}.$
All the results in the present paper readily extend to the orbifold
setting. \footnote{using, in particular, the uniform asymptotics for Bergman measures
on orbifolds in \cite[Thm 1.4]{d-l-m} as a replacement for the inequality
\ref{eq:upper bound on Bergman kernel}} For example, any Fano orbifold curve is of the form 
\[
X=\P^{1}/G
\]
 where $G$ is the finite group acting on $\P^{1}$ induced by the
action on $\C^{2}$ of a finite subgroup of $SU(2).$ By the ``ADE-trichotomy''
such groups fall into the three classes, corresponding to the classification
of simply laced Dynkin diagrams; two infinite series $A_{n}$and $D_{n}$
and three exceptional cases $E_{6},E_{7}$ and $E_{8}.$ As it turns
out, the ADE-trichotomy is detected by the corresponding partition
functions at the canonical value $\gamma=1$ (as follows from \cite[Thm 3.5]{berm12}): 
\begin{itemize}
\item (A)~$\mathcal{Z}_{N}(-1)=\infty$ for all $N$ (i.e. $\gamma_{k}(X)<1$
for all $k)$
\item (D)~$\mathcal{Z}_{N}(-1)<\infty$ for $N\gg1,$ but not all $N$
(i.e. $\gamma_{k}(X)>1$ for $k$ sufficiently large) 
\item (E)~$\mathcal{Z}_{N}(-1)<\infty$ for all $N$ (i.e. $\gamma_{k}(X)>1$
for all). 
\end{itemize}
Moreover, $\gamma_{k}(X)$ is strictly increasing wrt $k.$ On the
hand it can be shown that 
\[
\delta_{k}(X)=\delta(X)
\]
 and thus $\delta(X)=\gamma(X),$ while, $\gamma_{k}(X)<\delta_{k}(X).$ 

The notion of exceptionality has been extended to general Fano orbifolds
\cite{c-p-s}, motivated by the Minimal Model Program in birational
algebraic geometry \cite{sh}. A Fano orbifold $X$ is said to be
exceptional if $\alpha(X)>1,$ where $\alpha(X)$ denotes Tian's alpha-invariant
\cite{ti000} (which, in algebro-geometric terms, coincides with the
global log canonical threshold of $X$ \cite{dem}). For example,
in \cite[Cor 1.1]{c-p-s} a finite list of exceptional Fano orbifold
surfaces $X$ is given, realized as hypersurfaces in weighted three-dimensional
complex projective space. In general, it follows readily from the
definitions that 
\[
\alpha(X)\leq\gamma_{k}(X)
\]
(cf. \cite[Lemma 7.1]{berm11}). As a consequence, if $X$ is exceptional,
then $\mathcal{Z}_{N}$ is finite for \emph{any} $N.$ Does the converse
also hold? For Fano orbifold curves this is, indeed, the case, according
to the ADE-list above. 

\subsection{Acknowledgments}

Thanks to Bo Berndtsson for comments on a first draft of this manuscript.
This work was supported by grants from the Knut and Alice Wallenberg
foundation, the Göran Gustafsson foundation and the Swedish Research
Council.

\section{Proof of Theorem \ref{thm:Main intro}}

\subsection{Setup }

We will use additive notation for line bundles and metrics. Accordingly,
the $k$ the tensor power of a holomorphic line bundle $L$ over an
$n-$dimensional complex manifold $X$ will be denoted by $kL$ and
if $\phi$ is a metric on $L$ then $k\phi$ denotes the induced metric
on $kL.$ Accordingly, if $s$ is a holomorphic section of $L,$ i.e.
$s\in H^{0}(X,L),$ the point-wise norm of $s$ with respect to a
metric $\phi$ on $L$ is denoted by $|s|_{\phi}$ Given a local trivializing
section of $L$ we may identify $s$ with a local holomorphic function
on $X$ and $\phi$ with a local smooth function so that
\[
|s|_{\phi}^{2}:=|s|^{2}e^{-\phi}
\]
and the normalized curvature of the metric $\phi$ may be expressed
as 
\[
dd^{c}\phi:=\frac{i}{2\pi}\partial\bar{\partial}\phi.
\]
 A smooth metric $\phi$ on $L$ is said to have positive curvature
if $dd^{c}\phi>0$ and semi-positive curvature of $dd^{c}\phi\geq0$
(when identified with an $n\times n$ Hermitian matrix). Equivalently,
this means that, locally, $\phi$ is plurisubharmonic (psh) and strictly
psh, respectively. Given a metric $\phi$ with semi-positive curvature
we denote by $MA(\phi)$ the corresponding Monge-Ampère measure, normalized
to have unit total mass:

\[
MA(\phi):=\frac{1}{V}(dd^{c}\phi)^{n}.
\]

\subsubsection{\label{subsec:The-anti-canonical-setup}The anti-canonical setup}

Henceforth, the line bundle $L$ will be taken to be the anti-canonical
line bundle $-K_{X}$ of $X,$ i.e. top exterior power of the tangent
bundle of $X.$ Then any smooth metric $\phi$ on $-K_{X}$ induces
a volume form on $X$ that we shall, abusing notation slightly, denote
by $e^{-\phi}.$ This notation is intended to reflect the fact that
if $z_{1},...,z_{n}$ are local holomorphic coordinates on $X$ and
$\phi$ is locally represented by a function with respect to the local
trivialization $\partial/\partial z_{1}\wedge\cdots\partial/\partial z_{n}$
of $-K_{X},$ then the volume form in question has density $e^{-\phi}$
with respect to the local Euclidean volume form corresponding to $z_{1},...,z_{n}.$ 

Given a metric $\phi$ on $-K_{X}$ and a volume form $\mu$ on $X$
we shall denote by $H^{(k)}\left(\phi,\mu\right)$ the corresponding
Hermitian metric on the $N-$dimensional complex vector space $H^{0}(X,-kK_{X}),$
defined by 
\[
H^{(k)}\left(\phi,\mu\right)(s,s):=\int_{X}|s|_{k\phi}^{2}\mu.
\]
 The space of all metrics on $\phi$ on $-K_{X}$ with positive curvature
will be denoted by $\mathcal{H}.$ We will fix once and for all a
reference metric $\psi_{0}$ in $\mathcal{H}$ and a basis $s_{1}^{(k)},...,s_{N}^{(k)}$
in $H^{0}(X,-kK_{X})$ which is orthonormal with respect to the corresponding
Hermitian norm $H^{(k)}\left(\psi_{0},e^{-\psi_{0}}\right)$ (in the
notation used in Section \ref{sec:Introduction}, $e^{-\psi_{0}}=dV_{X}).$
Accordingly, we can identify a Hermitian metric $H$ on $H^{0}(X,-kK_{X})$
with the corresponding $N\times N$ positive definite Hermitian matrix
$H(s_{i}^{(k)},s_{j}^{(k)}).$ 

\subsubsection{Energies}

Following (essentially) the notation in \cite{b-b} we denote by $\mathcal{E}$
the functional on $\mathcal{H}$ uniquely determined by the following
conditions:

\[
d\mathcal{E}_{|\phi}=MA(\phi),\,\,\,\,\mathcal{E}(\psi_{0})=0
\]
Alternatively, $\mathcal{E}(\phi)$ may be explicitly defined by 
\[
\mathcal{E}(\phi):=\frac{1}{V(n+1)}\int_{X}\sum_{j=0}^{n}(\phi-\psi_{0})(dd^{c}\phi)^{n-j}\wedge(dd^{c}\psi_{0})^{j}
\]
Dually, following \cite{bbgz}, the pluricomplex energy of a probability
measure $\mu$ on $X$ (wrt the reference metric $\psi_{0})$ is defined
by
\[
E(\mu)=\sup_{\phi\in\mathcal{H}}\left(\mathcal{E}(\phi)-\int_{X}(\phi-\psi_{0})\mu\right)
\]
(however in \cite{bbgz} the functional $\mathcal{E}(\phi)$ is denoted
by $E(\phi)$ and the pluricomplex energy is denoted by $E^{*}$).
We will use the following basic
\begin{lem}
There exists a positive constant $c_{X}$ only depending on $X$ such
that
\begin{equation}
-\mathcal{E}(\phi)+\sup_{X}(\phi-\psi_{0})\leq nE\left(MA(\phi)\right)+c_{X}.\label{eq:bounding J type by E}
\end{equation}
\end{lem}

\begin{proof}
This is essentially well-known but for completeness a short proof
is provided. First observe that there exists a constant $c_{X}$ such
that $\sup_{X}(\phi-\psi_{0})-c_{X}$ is bounded from above by the
integral of $(\phi-\psi_{0})$ against $MA(\psi_{0}).$ Indeed, this
follows directly follows from the submean property of plurisubharmonic
functions and the compactness of $X.$ Hence, the proof is concluded
by invoking the following basic inequality (see \cite[Lemma 2.13]{berm6}):
\[
J(\phi):=-\mathcal{E}(\phi)+\int_{X}(\phi-\psi_{0})MA(\psi_{0})\leq nE\left(MA(\phi)\right)
\]
\end{proof}

\subsubsection{The twisted Ding and Mabuchi functional associated to $(\phi_{0},\gamma)$}

Given a volume form $dV$ on $X$ we will denote by $\phi_{0}$ the
corresponding metric on $-K_{X}$ (i.e. $dV=e^{-\phi_{0}}).$ To the
pair $(\phi_{0},\gamma)$ we attach the \emph{twisted Ding functional}
on $\mathcal{H}$ defined by 
\[
\mathcal{D}_{-\gamma}(\phi):=-\mathcal{E}(\phi)-\frac{1}{\gamma}\log\int_{X}e^{-\left(\gamma\phi+(1-\gamma)\phi_{0}\right)}.
\]
(coinciding with the ordinary Ding functional when $\gamma=-1).$
The definition is made so that $\mathcal{D}_{-\gamma}$ is scale invariant,
i.e. invariant under $\phi\mapsto\phi+c$ for any $c\in\R.$ The corresponding
(twisted) Mabuchi functional is usually defined, modulo an additive
constant, by demanding that its first variation is proportional to
the (twisted) scalar curvature \cite{ma,s-t}, but here it will be
convenient to use the thermodynamical formalism introduced in \cite[Prop 4.1]{berm6}: 

\begin{equation}
\mathcal{M}_{-\gamma}(\phi):=F_{-\gamma}(\mu),\,\,\,\mu=MA(\phi),\label{eq:M as free energy}
\end{equation}
 where $F_{\gamma}(\mu)$ is the \emph{free energy} of a probability
measure $\mu$ on $X$ defined by 
\begin{equation}
F_{-\gamma}(\mu):=-\gamma\left(E(\mu)+\int_{X}(\phi_{0}-\psi_{0})\mu\right)+\text{Ent \ensuremath{\left(\mu|e^{-\left(\gamma\psi_{0}+(1-\gamma)\phi_{0}\right)}\right)}},\label{eq:def of free en}
\end{equation}
 where $\text{Ent}\ensuremath{\left(\mu|\nu\right)}$ denotes the
entropy of a measure $\mu$ on $X$ relative to the measure $\nu$
on $X$ (using the sign convention that renders $\text{Ent}\ensuremath{\left(\mu|\nu\right)}$
non-negative when $\mu$ and $\nu$ are both probability measures).
By \cite[Prop 3.5]{berm6}

\begin{equation}
\mathcal{D}_{-\gamma}(\phi)\leq\gamma^{-1}\mathcal{M}_{-\gamma}(\phi)\label{eq:Ding smaller than Mab}
\end{equation}
(moreover, the two functionals $\mathcal{D}_{-\gamma}$ and $\mathcal{M}_{-\gamma}$
have the same infimum over $\mathcal{H},$ but his fact will not be
needed here).
\begin{rem}
In the notation of \cite{berm6}, $\mathcal{D}_{-\gamma}=-\mathcal{G}_{-\gamma}$
and the definition of the free energy $F_{-\gamma}$ used here is
$-\gamma$ times the definition employed in \cite{berm6}. When $\gamma=1$
formula \ref{eq:M as free energy} is equivalent to the Tian-Chen
formula for the Mabuchi functional \cite{ti0,ch0} and the case $\gamma\neq1$
is closely related to the generalized Mabuchi functional introduced
in \cite[Def 6.1]{s-t}.
\end{rem}

\subsection{Two inequalities}

The key new observation in the proof of Theorem \ref{thm:Main intro}
is the following proposition which yields a bound, from below, on
the partition function 

\[
\mathcal{Z}_{N,-\gamma}=\int_{X^{N}}\left\Vert \det S^{(k)}\right\Vert _{k\phi_{0}}^{-2\gamma/k}(e^{-\phi_{0}})^{\otimes N},
\]
in terms of the infimum over the space of all metrics $\phi$ on $-K_{X}$
of the functional $\mathcal{D}_{k,-\gamma}$ defined by 
\begin{equation}
\mathcal{D}_{k,-\gamma}(\phi):=-\mathcal{L}_{k}(\phi)-\frac{1}{\gamma}\log\int_{X}e^{-\left(\gamma\phi+(1-\gamma)\phi_{0}\right)},\label{eq:def of D k gamma on H}
\end{equation}
 where 
\[
\mathcal{L}_{k}(\phi):=-\frac{1}{N(k+\gamma)})\log\det H^{(k)}\left(\phi,e^{-\left(\gamma\phi+(1-\gamma)\phi_{0}\right)}\right).
\]
 The normalization have been chosen to ensure that 
\[
\mathcal{L}_{k}(\phi+c)=\mathcal{L}_{k}(\phi)+c\,\,\,,\forall c\in\R.
\]
As a consequence, since $\mathcal{L}_{k}(\phi)$ is increasing wrt
$\phi,$ its differential $d\mathcal{L}_{k|\phi}$ may be represented
by a probability measure on $X.$ For future reference we note that
the probability measure in question coincides with the \emph{Bergman
measure }associated to the Hermitian metric $H^{(k)}\left(\phi,e^{-\left(\gamma\phi+(1-\gamma)\phi_{0}\right)}\right):$
\begin{equation}
d\mathcal{L}_{k|\phi}=B_{k\phi}:=\rho_{k\phi}e^{-\left(\gamma\phi+(1-\gamma)\phi_{0}\right)},\,\,\,\rho_{k\phi}:=\frac{1}{N}\sum_{i=1}^{N_{k}}\left|S_{i}\right|_{k\phi}^{2}\label{eq:diff of L}
\end{equation}
 where $S_{i}$ denotes any bases in $H^{0}(X,-kK_{X})$ which is
orthonormal wrt $H^{(k)}\left(\phi,e^{-\left(\gamma\phi+(1-\gamma)\phi_{0}\right)}\right)$
(as follows from \cite[Lemma 2.1]{b-b}).
\begin{prop}
\label{prop:key inequal for part}Given $(\phi_{0},\gamma)$ the following
inequality holds for any $k$:
\[
-\frac{1}{\gamma N_{k}}\log\mathcal{Z}_{N_{k}}(-\gamma)\leq\left(1+\gamma k^{-1}\right)\inf_{\phi}\mathcal{D}_{k,-\gamma}(\phi)+\frac{1}{kN}\log N
\]
 where the infimum runs over all smooth metrics $\phi$ on $-K_{X}.$ 
\end{prop}

\begin{proof}
Let $\phi$ be a metric on $-K_{X}.$ Then we can rewrite
\[
\mathcal{Z}_{N,-\gamma}:=\int_{X^{N}}\left\Vert \det S^{(k)}\right\Vert _{k\phi_{0}}^{-2\gamma/k}(e^{-\phi_{0}})^{\otimes N}=\int_{X^{N}}\left\Vert \det S^{(k)}\right\Vert _{k\phi}^{-2\gamma/k}\left(e^{-(\gamma\phi+(1-\gamma)\phi_{0})}\right){}^{\otimes N}.
\]
 Indeed, locally on each factor of $X^{N}$ this simply amounts to
rewriting 
\[
\left(e^{-k\phi_{0}}\right)^{-\gamma/k}e^{-\phi_{0}}=\left(e^{-k\phi}\right)^{-\gamma/k}e^{-\phi_{0}}e^{-\left(\gamma\phi+(1-\gamma)\phi_{0}\right)}
\]
Now assume that $\phi$ has the property that $e^{-\left(\gamma\phi+(1-\gamma)\phi_{0}\right)}$
is a probability measure. Then, applying Hölder's inequality with
negative exponent $-\gamma/k$ (or Jensen's inequality applied to
the convex function $t\mapsto t^{-\gamma/k}$ on $]-\infty,\infty[$)
yields
\[
\mathcal{Z}_{N,-\gamma}\geq\left(\int_{X^{N}}\left\Vert \det S^{(k)}\right\Vert _{k\phi}^{2}\left(e^{-(\gamma\phi+(1-\gamma)\phi_{0})}\right){}^{\otimes N}\right)^{-\gamma/k}.
\]
 Taking logarithms this means that 
\[
-\frac{1}{\gamma N}\mathcal{\log Z}_{N,-\gamma}\leq\frac{1}{kN}\log\int_{X^{N}}\left\Vert \det S^{(k)}\right\Vert _{k\phi}^{2}\left(e^{-(\gamma\phi+(1-\gamma)\phi_{0})}\right){}^{\otimes N}.
\]
Now, for any metric $\phi$ on $-K_{X}$ we may apply the previous
inequality to $\phi+\log\int_{X}e^{-(\gamma\phi+(1-\gamma)\phi_{0})}$
and deduce that $-\frac{1}{\gamma N}\mathcal{\log Z}_{N,-\gamma}$
is bounded from above by
\[
(1+\gamma k^{-1})\left(\frac{1}{(k+\gamma)N}\log\int_{X^{N}}\left\Vert \det S^{(k)}\right\Vert _{k\phi}^{2}\left(e^{-(\gamma\phi+(1-\gamma)\phi_{0})}\right){}^{\otimes N}-\frac{1}{\gamma}\log\int_{X}e^{-(\gamma\phi+(1-\gamma)\phi_{0})}\right).
\]
The proof is thus concluded by invoking the following formula \cite[Lemma 5.3]{b-b},
which holds for any volume form $\mu$ on $X:$ 
\[
\int_{X^{N}}\left\Vert \det S^{(k)}\right\Vert _{k\phi}^{2}\mu^{\otimes N}=N!\det\left(H^{(k)}\left(\phi,\mu\right)\right)
\]
\end{proof}
We will also use the following slight generalization of the inequality
in \cite[formula 3.4]{b-f} (to the case $\gamma\neq1):$ 
\begin{lem}
\label{lem:bound on L funct}There exists a constant $C_{0}$ depending
only on $\psi_{0}$ such that the following inequality holds for $\phi^{(\epsilon)}:=\phi(1-\epsilon)+\epsilon\psi_{0}$
with $\epsilon:=\frac{\gamma-1}{k+\gamma}:$
\[
-\frac{1}{(1-\epsilon)}\mathcal{L}_{k}\left(\phi^{(\epsilon)}\right)\leq-\mathcal{E}(\phi)+C_{0}k^{-1}\left(-\mathcal{E}(\phi)+\sup_{X}(\phi-\psi_{0})\right)+\frac{\left|\gamma-1\right|}{k+1}\left\Vert \phi_{0}-\psi_{0}\right\Vert _{L^{\infty}(X)}.
\]
\end{lem}

\begin{proof}
This is shown in essentially the same way as in the proof of \cite[formula 3.4]{b-f},
but to pinpoint the exact dependence on the constant we recall the
argument. Let $\psi_{t}$ be a weak geodesic connecting $\phi$ (at
$t=1)$ with $\psi_{0}$ (at $t=0)$ \cite{ch00}. In particular,
this means that $t\mapsto\psi_{t}$ is a\emph{ psh path} (aka a subgeodesic)
in the following sense: extending $\psi_{t}$ to $X\times\left([0,1]\times i\R\right),$
so that $\psi_{t}$ is independent of the imaginary part of $t,$
the corresponding local function $(z,t)\mapsto\psi_{t}(z)$ is psh
locally on $X\times\left(]0,1[\times i\R\right).$ Moreover, it will
be convenient to use the following regularity properties \cite{ch00}:
$dd^{c}\psi_{t}\in L_{loc}^{\infty}$ for any fixed $t$ and $t\mapsto\psi_{t}$
is $C^{1}-$differentiable up to the boundary of $[0,1]$ (but, as
explained in \cite{b-f}, for the proof it is enough to use that $\psi_{t}$
is in $L_{loc}^{\infty}$ for any fixed $t$). Now,

\begin{equation}
\text{(i)}\,t\mapsto\mathcal{E}(\psi_{t})\,\text{is affine,\,\,\,\,\text{(ii)\,}t\ensuremath{\mapsto\mathcal{L}_{k}\left(\psi_{t}^{(\epsilon)}\right)}}\,\text{is\,concave}\label{eq:i and ii}
\end{equation}
 if $\epsilon$ is sufficiently small. In fact, the first statement
characterizes the geodesic $\phi_{t}$ among all psh paths $\phi_{t}$
as above \cite[Thm 1.7]{bbj} and the second one follows from \cite{bern0},
only using that $\psi_{t}^{(\epsilon)}$ is a psh path. To see this
rewrite $-kK_{X}=(k+1)L+K_{X}$ for $L=-K_{X}.$ Then, locally, rewriting
$e^{-k\phi}e^{-\left(\gamma\phi+(1-\gamma)\phi_{0}\right)}=e^{-\left(k\phi+\gamma\phi+(1-\gamma)\phi_{0}\right)}$
the Hermitian metric $H^{(k)}\left(\phi,e^{-\left(\gamma\phi+(1-\gamma)\phi_{0}\right)}\right)$
coincides with the $L^{2}-$metric on $H^{0}(X,(k+1)L+K_{X})$ induced
by the metric $k\phi+\gamma\phi+(1-\gamma)\phi_{0}$ on $(k+1)+L_{X}.$
Accordingly, $\mathcal{L}_{k}\left(\phi\right)$ may be identified
with the $L^{2}-$metric on the determinant line of $H^{0}(X,(k+1)L+K_{X}).$
Now replace $\phi$ with $\psi_{t}^{(\epsilon)}$ and decompose the
corresponding metric on $(k+1)L+K_{X}$ as 

\begin{equation}
k\psi_{t}^{(\epsilon)}+\gamma\psi_{t}^{(\epsilon)}+(1-\gamma)\phi_{0}=(k+\gamma)(1-\epsilon)\psi_{t}+\left((k+\gamma)\epsilon\psi_{0}+(1-\gamma)\phi_{0}\right).\label{eq:decomp of metrics}
\end{equation}
 The second term above has non-negative curvature on $X$ if $\epsilon$
is sufficiently large. For simplicity, we will first consider the
special case that $\phi_{0}=\psi_{0}.$ Then the non-negativity in
question holds if $\epsilon\geq\frac{\gamma-1}{k+\gamma}.$ Henceforth
will assumed that $\epsilon=\frac{\gamma-1}{k+\gamma}$ (then $(1-\epsilon)=(k+1)(/k+\gamma)>0).$
Since $\psi_{t}$ is locally psh on $X\times\left([0,1]\times i\R\right)$
the whole expression in formula \ref{eq:decomp of metrics} is thus
locally psh. Hence, the convexity of $t\mapsto\mathcal{L}_{k}\left(\psi_{t}^{(\epsilon)}\right)$
follows from the positivity of direct image bundles in \cite{bern0},
applied the to trivial fibration $X\times\left(]0,1[\times i\R\right)\rightarrow]0,1[\times i\R.$
This concludes the proof of the properties in formula \ref{eq:i and ii}.
As a consequence, the function $t\mapsto-\frac{1}{(1-\epsilon)}\mathcal{L}_{k}\left(\psi_{t}^{(\epsilon)}\right)+\mathcal{E}(\psi_{t})$
is concave, giving, 
\[
-\frac{1}{(1-\epsilon)}\mathcal{L}_{k}\left(\phi^{(\epsilon)}\right)+\mathcal{E}(\phi)\leq-\mathcal{L}_{k}\left(\psi_{0}^{(\epsilon)}\right)+\mathcal{E}(\psi_{0})+\int_{X}\left(-\frac{d\psi_{t}}{dt}|_{t=0}\right)\left(\frac{1}{(1-\epsilon)}d\mathcal{L}_{k}\left(\phi^{(\epsilon)}\right)-d\mathcal{E}\right)_{|\psi_{0}}
\]
Now assume first that $\phi$ is sup-normalized, i.e. that $\sup_{X}(\phi-\psi_{0})=0.$
Then it follows from the convexity of $t\mapsto\phi_{t}$ that $\frac{d\psi_{t}}{dt}|_{t=0}\leq0.$
Next, since we are considering the special case $\phi_{0}=\psi_{0}$
and $\psi_{0}^{(\epsilon)}=\psi_{0}$ the term $\mathcal{L}_{k}\left(\psi_{0}^{(\epsilon)}\right)$
vanishes and so does $\mathcal{E}(\psi_{0})$ (by definition). Moreover,
since the differential of the functional 
\[
\phi\mapsto\frac{1}{(1-\epsilon)}\mathcal{L}_{k}\left(\phi^{(\epsilon)}\right)
\]
 is given by the Bergman measure $B_{k}$ associated to the Hermitian
metric $H^{(k)}(\psi_{0},e^{-\psi_{0}})$ (by formula \ref{eq:diff of L})
it follows from Bergman kernel asymptotics \cite{ti0} that there
exists a constant $C_{0}$ (depending only on $\psi_{0})$ such that
\begin{equation}
\left(\frac{1}{(1-\epsilon)}d\mathcal{L}_{k}\left(\phi^{(\epsilon)}\right)-d\mathcal{E}\right)_{|\psi_{0}}\leq-C_{0}k^{-1}d\mathcal{E}(\psi_{0}).\label{eq:upper bound on Bergman kernel}
\end{equation}
(in fact only an \emph{upper }bound on $B_{k}$ is needed for which
there is an elementary proof \cite[Prop 2.4]{b-f}). Hence,
\[
-\frac{1}{(1-\epsilon)}\mathcal{L}_{k}\left(\phi^{(\epsilon)}\right)+\mathcal{E}(\phi)\leq C_{0}k^{-1}\int\left(-\frac{d\psi_{t}}{dt}|_{t=0}\right)\left(d\mathcal{E}\right)_{|\psi_{0}}=-C_{0}k^{-1}\mathcal{E}(\phi),
\]
 using, in the last equality $(i)$ in formula \ref{eq:i and ii}.
Replacing a general $\phi\in\mathcal{H}$ with its sup-normalized
version $\phi-\sup_{X}(\phi-\psi_{0})$ we deduce that
\[
-\frac{1}{(1-\epsilon)}\mathcal{L}_{k}\left(\phi^{(\epsilon)}\right)+\mathcal{E}(\phi)\leq C_{0}k^{-1}\left(-\mathcal{E}(\phi)+\sup_{X}(\phi-\psi_{0})\right).
\]
This concludes the proof when $\phi_{0}=\psi_{0}.$ Finally, to handle
the case of a general case note that replacing $\phi_{0}$ with $\psi_{0}$
in the definition of $\mathcal{L}_{k}\left(\phi^{(\epsilon)}\right)$
just gives rise to an extra term which, after multiplication by $\frac{1}{(1-\epsilon)},$
may estimated from above by
\[
\frac{1}{(1-\epsilon)}\frac{1}{k+\gamma}\log e^{(\gamma-1)\sup_{X}(\phi_{0}-\psi_{0})}\leq\sup_{X}\left|\phi_{0}-\psi_{0}\right|\frac{1}{(1-\epsilon)(k+\gamma)}\left|\gamma-1\right|=\sup_{X}\left|\phi_{0}-\psi_{0}\right|\frac{1}{k+1}\left|\gamma-1\right|.
\]
\end{proof}

\subsection{\label{subsec:Conclusion-of-the}Conclusion of the proof of Theorem
\ref{thm:Main intro}}

By Prop \ref{prop:key inequal for part} the following inequality
holds for any metric $\phi$ on $-K_{X}$ and number satisfying $(1-\epsilon)\geq0:$
\[
-\frac{1}{\gamma N}\log\mathcal{Z}_{N}(-\gamma)\leq\left(1+\gamma k^{-1}\right)(1-\epsilon)\left(\frac{1}{N(k+\gamma)(1-\epsilon)}\mathcal{L}_{k}\left(\phi\right)-\frac{1}{\gamma(1-\epsilon)}\log\int_{X}e^{-\left(\gamma\phi+(1-\gamma)\phi_{0}\right)}\right)
\]
Taking $\epsilon=\frac{\gamma-1}{k+\gamma}$ and replacing $\phi$
with $\phi^{(\epsilon)}$ (defined as in the previous lemma) and setting
$\gamma^{(\epsilon)}:=(1-\epsilon)\gamma$ thus yields (using that
$\left(1+\gamma k^{-1}\right)(1-\epsilon)=1+k^{-1})$)
\[
-\frac{1}{\gamma N}\log\mathcal{Z}_{N}(-\gamma)\leq(1+k^{-1})\left(\frac{1}{N(k+\gamma)(1-\epsilon)}\log\det H^{(k)}(\phi^{(\epsilon)},\gamma)-\frac{1}{\gamma^{(\epsilon)}}\log\int_{X}e^{-\left(\gamma^{(\epsilon)}\phi+(1-\gamma^{(\epsilon)})\phi_{0}\right)}\right)
\]
Next, in order to fix ideas, we first consider the special case when
$\phi_{0}=\psi_{0}.$ Then, by the previous lemma, the right hand
side above is bounded from above by
\begin{equation}
(1+k^{-1})\left(\mathcal{D}_{-\gamma^{(\epsilon)}}(\phi)+C_{0}k^{-1}\left(-\mathcal{E}(\phi)+\sup_{X}(\phi-\psi_{0})\right)\right).\label{eq:pf main theorem a}
\end{equation}
Since (trivially) $\mathcal{D}_{-\gamma^{(\epsilon)}}(\phi)\leq-\mathcal{E}(\phi)+\sup_{X}(\phi-\psi_{0})$
it follows that 
\[
-\frac{1}{\gamma N}\log\mathcal{Z}_{N}(-\gamma)\leq\mathcal{D}_{-\gamma^{(\epsilon)}}(\phi)+(C_{0}+1)k^{-1}\left(-\mathcal{E}(\phi)+\sup_{X}(\phi-\psi_{0})\right)
\]
Invoking the inequality \ref{eq:bounding J type by E} thus reveals
that there exists a constant $C$ only depending on $\psi_{0}$ such
that 
\begin{equation}
-\frac{1}{N}\log\mathcal{Z}_{N}(-\gamma)\leq\gamma\mathcal{D}_{-\gamma^{(\epsilon)}}(\phi)+C\gamma k^{-1}E\left(MA(\phi)\right)+C\gamma k^{-1}\label{eq:pf theorem main b}
\end{equation}
Next, we rewrite the first two terms in the right hand side above
as 
\[
\gamma\mathcal{D}_{-\gamma^{(\epsilon)}}(\phi)+C\gamma k^{-1}E\left(MA(\phi)\right)=(1-\epsilon)^{-1}\left(\gamma^{(\epsilon)}\mathcal{D}_{-\gamma^{(\epsilon)}}(\phi)+C\gamma^{(\epsilon)}k^{-1}E\left(MA(\phi)\right)\right),
\]
 where the second factor above is is bounded from above by $\mathcal{M}_{-\gamma^{(\epsilon)}(1-Ck^{-1})},$
as follows from the inequality \ref{eq:Ding smaller than Mab} and
the free energy formula \ref{eq:M as free energy}. Hence, 
\[
-\frac{1}{N}\log\mathcal{Z}_{N}(-\gamma)\leq(1-\epsilon)^{-1}\mathcal{M}_{-\gamma^{(\epsilon)}(1-Ck^{-1})}+C\gamma k^{-1},
\]
 proving the theorem in the case when $\phi_{0}=\psi_{0}.$ The general
case is shown in essentially the same way, by first including the
error term involving $\phi_{0}$ from Lemma \ref{lem:bound on L funct}
in formula \ref{eq:pf main theorem a} and then, in formula \ref{eq:pf theorem main b},
estimating 
\[
E\left(MA(\phi)\right)\leq\left(E\left(MA(\phi)\right)+\int(\phi_{0}-\psi_{0})MA(\phi)\right)+\left\Vert \phi_{0}-\psi_{0}\right\Vert _{L^{\infty}(X)},
\]
so that the first term in the right hand side above can be absorbed
into the twisted Mabuchi functional, as before.

\subsection{\label{subsec:The-case gamma smaller}The case $\gamma\protect\leq1$}

For $\gamma\leq1$ the estimate in Theorem \ref{thm:Main intro} implies
that (after perhaps increasing the constant $C):$
\[
-\frac{1}{N}\log\mathcal{Z}_{N}(-\gamma)\leq\inf_{\mathcal{H}}\mathcal{M}_{-\gamma(1-Ck^{-1})}+Ck^{-1}+k^{-1}\gamma\left(|1-\gamma|+C\right)\log\left\Vert \frac{dV}{dV_{X}}\right\Vert _{L^{\infty}(X)}.
\]
 Indeed, by a simple scaling argument (applied to Lemma \ref{lem:bound on L funct})
it it enough to consider the case when $e^{-\left(\gamma\psi_{0}+(1-\gamma)\phi_{0}\right)}$
is a probability measure. This implies that the entropy term in the
free energy $F_{\gamma}$ is non-negative. It then follows readily
from the definition that the function 
\[
T\mapsto\inf TF_{T^{-1}}
\]
 is increasing in $T$ (where the infimum is taken over a given subset
of $\mathcal{P}(X)).$ In particular, applied to the present setup
at $T_{0}=(k+\gamma)/(k+1)$ and $T_{1}=1$ this monotonicity yields
(since $T_{0}\leq T_{1}$ when $\gamma\leq1)$ 
\[
\frac{k+\gamma}{k+1}\inf_{\mathcal{H}}\mathcal{M}_{-\gamma(1-Ck^{-1})(k+1)/(k+\gamma)}\leq\inf_{\mathcal{H}}\mathcal{M}_{-\gamma(1-Ck^{-1})},
\]
 as desired. 

\section{Comparison with the quantization approach}

Given a holomorphic line bundle $L$ over a compact complex manifold
$X$ and a positive integer $k$ denote by $\mathcal{H}_{k}(L)$ the
space of all Hermitian metrics on the $N-$complex vector space $H^{0}(X,kL).$
The ``Fubini-Study map'' $FS$ maps $\mathcal{H}_{k}(L)$ into the
space $\mathcal{H}(L)$ of all metrics on $L$ with positive curvature:
\begin{equation}
FS:\mathcal{H}_{k}(L)\rightarrow\mathcal{H}(L),\,\,\,\,FS(H):=k^{-1}\log\left(\frac{1}{N}\sum_{i=1}^{N}\left|s_{i}^{H}\right|^{2}\right)\label{eq:def of fs}
\end{equation}
where $(s_{i}^{H})$ is any basis in $H^{0}(X,kL)$ which is orthonormal
wrt $H.$ The normalization by $N$ used here is non-standard, but
it will simplify some of the formulas below. 

Henceforth, we shall specialize to the anti-canonical setting in Section
\ref{subsec:The-anti-canonical-setup}. Thus $X$ is a Fano manifold
and $L=-K_{X}.$ We will abbreviate $\mathcal{H}_{k}(-kK_{X})=\mathcal{H}_{k}$
and $\mathcal{H}(-K_{X})=\mathcal{H}.$ Consider now the functional
$D_{k,-\gamma}$ on $\mathcal{H}_{k}$ defined by 
\begin{equation}
D_{k,-\gamma}(H_{k}):=\frac{1}{kN_{k}}\log\det H^{(k)}-\frac{1}{\gamma}\log\int_{X}e^{-\left(\gamma FS(H_{k})+(1-\gamma)\phi_{0}\right)},\label{eq:def of D k gamma on H k}
\end{equation}
 which is invariant under scaling by positive numbers: 
\[
D_{k,-\gamma}(e^{c}H_{k})=D_{k,-\gamma}(H_{k})\,\,\,\forall c\in\R.
\]
As is well-known the functional $D_{k,-\gamma}$ on $\mathcal{H}_{k}$
can be viewed as a quantization of the functional $\mathcal{D}_{-\gamma}$
on $\mathcal{H}$ \cite{bbgz,rtz}. The following proposition relates
the functional $D_{k,-\gamma}$ on $\mathcal{H}_{k}$ to the functional
$\mathcal{D}_{k,-\gamma}$ on $\mathcal{H}$ defined in formula \ref{eq:def of D k gamma on H}.
\begin{prop}
\label{prop:inequal for quantized Ding}For any metric $\phi$ on
$-K_{X}$
\[
D_{k,-\gamma}\left(H^{(k)}\left(\phi,e^{-\left(\gamma\phi+(1-\gamma)\phi_{0}\right)}\right)\right)\leq\left(1+\gamma k^{-1}\right)\mathcal{D}_{k,-\gamma}(\phi)
\]
and for any $H\in\mathcal{H}_{k}$
\[
\left(1+\gamma k^{-1}\right)\mathcal{D}_{k,-\gamma}\left(FS(H)\right)\leq D_{k,-\gamma}(H)
\]
 In particular, 
\[
\inf_{\mathcal{H}_{k}}D_{k,-\gamma}=\left(1+\gamma k^{-1}\right)\inf_{\mathcal{H}}\mathcal{D}_{k,-\gamma}
\]
\end{prop}

\begin{proof}
To prove the first inequality let $\phi$ be a given metric on $-K_{X}$
and set $\psi_{k}:=FS\left(H^{(k)}\left(\phi,e^{-\left(\gamma\phi+(1-\gamma)\phi_{0}\right)}\right)\right).$
Then

\begin{equation}
\int_{X}e^{-\left(\gamma\psi_{k}+(1-\gamma)\phi_{0}\right)}\geq\left(\int_{X}e^{-\left(\gamma\phi+(1-\gamma)\phi_{0}\right)}\right)^{(1+\gamma/k)}.\label{eq:first ineq in pf inf is inf}
\end{equation}
 Indeed, rewriting $e^{-\left(\gamma\psi_{k}+(1-\gamma)\phi_{0}\right)}=e^{-\gamma\left(\psi_{k}-\phi\right)}e^{-\left(\gamma\phi+(1-\gamma)\phi_{0}\right)}$
and using that 
\[
e^{\left(\psi_{k}-\phi\right)}=\rho_{k\phi}
\]
(as follows directly from the definition of $\rho_{k\phi}$ in formula
\ref{eq:diff of L}) gives 
\[
\int_{X}e^{-\left(\gamma\psi_{k}+(1-\gamma)\phi_{0}\right)}=\int_{X}(\rho_{k\phi})^{-\gamma/k}e^{-\left(\gamma\phi+(1-\gamma)\phi_{0}\right)}\geq
\]
\[
\left(\int_{X}(\rho_{k\phi})e^{-\left(\gamma\phi+(1-\gamma)\phi_{0}\right)}\right)^{-\gamma/k}\left(\int_{X}e^{-\left(\gamma\phi+(1-\gamma)\phi_{0}\right)}\right)^{(1+\gamma/k)}
\]
 using Hölder's inequality with negative exponent $-\gamma/k$ (or
Jensen's inequality applied to the convex function $t\mapsto t^{-\gamma/k}$
on $]-\infty,\infty[$). The integral appearing in the first factor
in the right hand side above is precisely the integral of the Bergman
measure $B_{k\phi}$ (defined in formula \ref{eq:diff of L}) and
thus equal to one, which proves the inequality \ref{eq:first ineq in pf inf is inf}.
Hence, using $(1+\gamma k^{-1})/(k+\gamma)=1/k,$ 
\[
(1+\gamma/k)\mathcal{D}_{k,-\gamma}\left(H^{(k)}\left(\phi,e^{-\left(\gamma\phi+(1-\gamma)\phi_{0}\right)}\right)\right)\leq
\]
\[
\frac{1}{kN}\log\det H^{(k)}\left(\phi,e^{-\left(\gamma\phi+(1-\gamma)\phi_{0}\right)}\right)-\gamma^{-1}\log\left(\int_{X}e^{-\left(\gamma\phi+(1-\gamma)\phi_{0}\right)}\right),
\]
 which proves the first inequality stated in the proposition. To prove
the second one first observe that for any $H$ and and volume form
$\mu$ on $X$ 
\begin{equation}
\det\left(H^{(k)}\left(FS(H),\mu\right)\right)\leq\det H\cdot(\int_{X}\mu)^{N}\label{eq:second ineq in pf inf is inf}
\end{equation}
 Indeed, for any given $\phi$ in $\mathcal{H}$ and $H\in\mathcal{H}_{k},$
taking a basis $(s_{i}^{H})$ in $H^{0}(X,-kK_{X})$ which is orthonormal
wrt $H,$ we can factorize 
\begin{equation}
\det\left(H^{(k)}\left(\phi,\mu\right)\right)=\det H\cdot\det\left(H^{(k)}\left(\phi,\mu\right)(s_{i}^{H},s_{j}^{H})\right),\label{eq:factoriz}
\end{equation}
 where the second factor arises as the determinant of the change of
bases matrix between the reference basis $(s_{i}^{(k)})$ in $H^{0}(X,-kK_{X})$
and $(s_{i}^{H}).$ Next, by the arithmetic/geometric means inequality
\[
\left(\det\left(H^{(k)}\left(\phi,\mu\right)(s_{i}^{H},s_{j}^{H})\right)\right)^{1/N}\leq N^{-1}\sum_{i=1}^{N}H^{(k)}\left(\phi,\mu\right)(s_{i}^{H},s_{i}^{H}).
\]
 Now assume that $\phi=FS(H).$ Then 
\[
H^{(k)}\left(\phi,\mu\right)(s_{i}^{H},s_{i}^{H}):=\int_{X}\frac{|s_{i}^{H}|^{2}}{N^{-1}\sum_{j=1}^{N}|s_{i}^{H}|^{2}}\mu.
\]
 Hence, the second factor in the right hand side in formula \ref{eq:factoriz}
is bounded from above by the $N$th power of the integral of $\mu,$
proving the inequality \ref{eq:second ineq in pf inf is inf}. Thus,
if $H$ is a given element in $\mathcal{H}_{k}$ which is normalized
so that $\int e^{-\left(\gamma FS(H)+(1-\gamma)\phi_{0}\right)}=1,$
then applying the inequality \ref{eq:second ineq in pf inf is inf}
to $\mu=e^{-\left(\gamma FS(H)+(1-\gamma)\phi_{0}\right)}$ proves
the second inequality for any normalized $H$ in $\mathcal{H}_{k}.$
Finally, since both sides of the inequality in question are invariant
under scaling, $H\rightarrow e^{c}H,$ this concludes the proof for
a general $H$ in $\mathcal{H}_{k}.$
\end{proof}
\begin{rem}
The previous proposition refines a monotonicity result \cite[Lemma 2.6]{ber},
concerning Donaldson's iteration in the anti-canonical setting of
\cite{do1c}. Indeed, applying the first inequality to $\phi=FS(H)$
and then the second inequality reveals that $D_{k,-\gamma}(H)$ is
decreasing under Donaldson's map on $\mathcal{H}_{k},$ defined as
the composition of the maps $F\mapsto FS(H)$ and $\phi\mapsto H^{(k)}\left(\phi,e^{-\left(\gamma\phi+(1-\gamma)\phi_{0}\right)}\right).$
As in \cite{ber} one gets equality in the first equality in the proposition
when $\phi=FS(H)$ iff $H$ is a balanced metric in $\mathcal{H}_{k}$
in the anti-canonical sense of \cite{do1c,bbgz,rtz}. 
\end{rem}

Combining Prop \ref{prop:key inequal for part} and the equality for
the infima in Prop \ref{prop:inequal for quantized Ding} (only the
upper bound is needed) we thus arrive at the following result:
\begin{thm}
\label{thm:bound on Z in terms of D on H k}Let $X$ be a Fano manifold.
Given $(\phi_{0},\gamma)$ the following inequality holds for any
$k$:
\[
-\frac{1}{\gamma N}\log\mathcal{Z}_{N_{k}}(-\gamma)\leq\inf_{\mathcal{H}_{k}}D_{k,-\gamma}+\frac{1}{kN}\log N.
\]
\end{thm}

As a consequence, if $\mathcal{Z}_{N_{k}}(-\gamma)$ is finite, then
the infimum of $D_{k,-\gamma}$ over $\mathcal{H}_{k}$ is finite.
In other words, the invariant $\gamma_{k}(X)$ defined by formula
\ref{eq:def of gamma k and gamma} is smaller than or equal to the
coercivity threshold of the functional $D_{k}$ on $\mathcal{H}_{k}$
which, by \cite{rtz}, coincides with the invariant $\delta_{k}(X)$
introduced in \cite{f-o} (appearing in formula \ref{eq:gamma k smaler then delta k}).
We thus arrive at a new proof of the following inequality first shown
in\cite{fu2} (see \ref{eq:gamma k smaler then delta k} for a reformulation
of the proof in terms of non-Archimedean pluripotential theory).
\begin{cor}
\cite[Thm 2.5]{f-o}. For a Fano manifold $X$ the following inequality
holds:
\[
\gamma_{k}(X)\leq\delta_{k}(X)
\]
\end{cor}

Combining the equality for the infima in Prop \ref{prop:inequal for quantized Ding}
with the argument employed in Section \ref{subsec:Conclusion-of-the}
also yields the following analog of Theorem \ref{thm:Main intro}:
\begin{thm}
There exists a constant $C>0$ (depending only on the reference volume
for $dV_{X})$ such that for any $\gamma>0$ and positive integer
$k$ 
\[
\inf_{\mathcal{H}_{k}}D_{k,-\gamma}\leq-\frac{1}{N}\log\mathcal{Z}_{N}(-\gamma)\leq\frac{k+\gamma}{k+1}\inf_{\mathcal{H}}\mathcal{M}_{-\gamma c_{k}}+k^{-1}\gamma\left(C+\left(|1-\gamma|+C\right)\log\left\Vert \frac{dV}{dV_{X}}\right\Vert _{L^{\infty}(X)}\right),
\]
 where $c_{k}:=(1-Ck^{-1})(k+1)/(k+\gamma).$
\end{thm}

As explained in Section \ref{subsec:Comparison-with-the qu} this
inequality is closely related to results in \cite{zh}.


\begin{thebibliography}{10}
\bibitem{au}Aubin, T: Equations du type Monge-Amp`ere sur les varietes
Kahleriennes compactes. Bull. Sci. Math. (2) 102 (1978), no. 1, 63--95

\bibitem{berm6}Berman, R.J: A thermodynamical formalism for Monge-Ampere
equations, Moser-Trudinger inequalities and Kahler-Einstein metrics\emph{.}
Advances in Math. 1254. Volume: 248. 2013

\bibitem{ber}Berman, R.J: Relative K\"{ }ahler-Ricci flows and their
quantization, Analysis and PDE, 6 (1), (2013), 131-180.

\bibitem{berm8}Berman, R.J: Large deviations for Gibbs measures with
singular Hamiltonians and emergence of Kähler-Einstein metrics. Communications
in Mathematical Physics. Volume 354, Issue 3, pp 1133--1172 (2017)

\bibitem{berm8 comma 5} Berman, R.J: Kähler-Einstein metrics, canonical
random point processes and birational geometry.\emph{ }Proceedings
of Symposia in Pure Mathematics. Volume 97.1 : Algebraic Geometry
Salt Lake City 2015 (Part 1). pp 29-74 

\bibitem{berm11}Berman, R.J: An invitation to K\"{ }ahler-Einstein
metrics and random point processes. Surveys in Differential Geometry
Volume 23 (2018) Pages: 35 -- 87

\bibitem{berm11b}Berman, R.J: Emergent complex geometry. https://arxiv.org/abs/2109.00307. 

\bibitem{berm12}Berman, R.J.: Kähler-Einstein metrics and Archimedean
zeta functions. Preprint

\bibitem{b-b}Berman, R.J.; Boucksom, S: Growth of balls of holomorphic
sections and energy at equilibrium. Invent. Math. Vol. 181, Issue
2 (2010), p. 337

\bibitem{bbgz}Berman, R.J; Boucksom, S; Guedj,V; Zeriahi: A variational
approach to complex Monge-Ampere equations. Publications math. de
l'IHÉS (2012): 1-67 , November 14, 2012

\bibitem{bbegz}R.J.Berman\emph{; }Eyssidieu, P: S. Boucksom, V. Guedj,
A. Zeriahi: Kähler-Einstein metrics and the Kähler-Ricci flow on log
Fano varieties. Journal fur die Reine und Angewandte Mathematik (published
on-line 2016).

\bibitem{b-f}R. J. Berman and G. Freixas i Montplet: An arithmetic
Hilbert-Samuel theorem for singular hermitian line bundles and cusp
forms. Compos. Math., 150(10):1703--1728, 2014

\bibitem{bbj} Berman, R.J; Boucksom, S; Jonsson, M: A variational
approach to the Yau-Tian-Donaldson conjecture. J. Amer. Math. Soc.
34(2021), 605--652. arXiv:1509.04561 

\bibitem{bern0}B. Berndtsson, Curvature of vector bundles associated
to holomorphic fibrations, Ann. of Math. (2) 169 (2009), no. 2, 531--560

\bibitem{bern}B. Berndtsson. Probability measures related to geodesics
in the space of K\"{ }ahler metrics, 2009. arXiv:0907.1806

\bibitem{bl-j}Blum, H.; Jonsson, M.; Thresholds, valuations, and
K-stability. Adv. Math. 365 (2020).

\bibitem{c-p-s}I Cheltsov, J Park, C Shramov: Exceptional del Pezzo
hypersurfaces. Journal of Geometric Analysis, 2010 - Springer

\bibitem{ch00}X.X. Chen, The space of K\"{ }ahler metrics, J. Differential
Geom. 56(2) (2000) 189--234, MR 1863016, Zbl 1041.5800

\bibitem{ch0}Chen, X.X: On the lower bound of the Mabuchi energy
and its application , Int. Math. Res.Not. 2000, no. 12, 607-623

\bibitem{c-d-s}X.X Chen, S. Donaldson, S. Sun. K\textasciidieresis ahler-Einstein
metrics on Fano manifolds, I, II, III. J. Amer. Math. Soc. 28 (2015).

\bibitem{d-l-m}X. Dai, K. Liu, and X. Ma, On the asymptotic expansion
of Bergman kernel. J. Differential Geom. 72(2006), 1-- 41

\bibitem{d-s}V. Datar, G. Sz\'{ }ekelyhidi. K\"{ }ahler--Einstein
metrics along the smooth continuity method. Geom. Funct. Anal. 26
(2016), 975-1010

\bibitem{d-z}Dembo, A; Zeitouni O: Large deviation techniques and
applications. Jones and Bartlett Publ. 1993

\bibitem{dem}Demailly, J-P: Appendix to I. Cheltsov and C. Shramov\textquoteright s
article \textquotedblleft Log canonical thresholds of smooth Fano
threefolds\textquotedblright . Uspekhi Mat. Nauk, 63:5(383) (2008),
73--180

\bibitem{ding}Ding, W.: Remarks on the existence problem of positive
Kähler-Einstein metrics. Math. Ann.463--472 (1988).

\bibitem{do0}Donaldson, S. K. (2001). Planck\textquoteright s constant
in complex and almost-complex geometry. XIIIth International Congress
on Mathematical Physics (London, 2000), 63--72, Int. Press, Boston

\bibitem{do1}Donaldson, S. K: Scalar curvature and projective embeddings.
I. J. Differential Geom. 59 (2001), no. 3, 479--522

\bibitem{do1b}Donaldson, S. K. Scalar curvature and projective embeddings.
II. Q. J. Math. 56 (2005), no. 3, 345--356. 

\bibitem{do1c}S. K. Donaldson,Some numerical results in complex differential
geometry, Pure and ApplMath Quartly,5(2009), 571--618

\bibitem{f-k-z}Ferrari, F., Klevtsov, S., Zelditch, S.: Random Kähler
metrics. Nucl. Phys. B 869(1), 89--110 (2013

\bibitem{fu}Fujita,K: On Berman--Gibbs stability and K-stability
of Q -Fano varieties. Compositio Mathematica , Volume 152 , Issue
2 , February 2016 , pp. 288 - 298 

\bibitem{f-o}Fujita, K; Odaka, Y: On the K-stability of Fano varieties
and anticanonical divisors. Tohoku Math. J. (2) 70 (2018), no. 4,
511--521.

\bibitem{fu2}Fujita, K: A valuative criterion for uniform K-stability
of Q-Fano varieties. J. Reine Angew. Math. 751 (2019), 309--338. 

\bibitem{l-x-z}Y Liu, C Xu, Z Zhuang: Finite generation for valuations
computing stability thresholds and applications to K-stability. Preprint
at arXiv:2102.09405 (2021)

\bibitem{ma}Mabuchi, T: K-energy maps integrating Futaki invariants.
Tohoku Math. J. (2) 38 (1986),no. 4, 575--593

\bibitem{rtz}Y. A. Rubinstein, G. Tian, and K. Zhang. Basis divisors
and balanced metrics, 2020. arXiv:2008.08829, to appear in J.Reine
Angew. Math.

\bibitem{sh} VV Shokurov: Complements on surfaces. Journal of Mathematical
Sciences. Vol. 102, No.2, 2000

\bibitem{s-t}Song, Y; Tian, G: Canonical measures and Kähler-Ricci
flow. J. Amer. Math. Soc. 25 (2012), no. 2, 303--353. 

\bibitem{sz}G. Sz\'{ }ekelyhidi. Greatest lower bounds on the Ricci
curvature of Fano manifolds. Compos. Math. 147 (2011), 319--331

\bibitem{ti000}G. Tian, On Kähler--Einstein metrics on certain Kähler
manifolds with $c_{1}(M)>0.$ Inventiones Mathematicae 89 (1987),
225--246

\bibitem{ti00}G. Tian, On a set of polarized Kahler metrics on algebraic
manifolds, J. Differential Geom. 32 (1990), 99--130.

\bibitem{ti0}G. Tian Transcendental Methods in Algebraic Geometry.
Lecture Notes in Math., vol. 1646, Cetraro, 1994 (1996), pp. 143-185

\bibitem{ti}Tian, G: K\"{ }ahler-Einstein metrics with positive scalar
curvature, Invent. Math.130(1997),no. 1, 1--37

\bibitem{y2}Yau, S. T.: Nonlinear analysis in geometry. Enseign.
Math. (2) 33 (1987), no. 1-2, 109--158. 58-02

\bibitem{y}Yau, S-T: On the Ricci curvature of a compact Kähler manifold
and the complex Monge-Ampère equation. I. Comm. Pure Appl. Math. 31
(1978), no. 3, 339--411

\bibitem{zh}K. Zhang: A quantization proof of the uniform Yau-Tian-Donaldson
conjecture. preprint arXiv:2102.02438 (2021)
\end{thebibliography}
\end{document}